\documentclass[12pt,a4paper]{amsart}
\usepackage{graphicx}
\usepackage{mathptmx}
\usepackage{xcolor}
\usepackage[top=35mm, bottom=35mm, left=30mm, right=30mm]{geometry}
\usepackage{amsmath,amssymb,mathrsfs}
\usepackage[hidelinks]{hyperref}
\newtheorem{thm}{Theorem}[section]
\newtheorem{cor}[thm]{Corollary}
\newtheorem{lem}[thm]{Lemma}
\newtheorem*{Claim1}{Claim 1}
\newtheorem*{Claim2}{Claim 2}
\theoremstyle{definition}
\newtheorem{defn}[thm]{Definition}
\newtheorem{rem}[thm]{Remark}
\newtheorem{question}[thm]{Question}
\numberwithin{equation}{section}
\numberwithin{figure}{section}
\DeclareMathOperator{\supp}{supp}
\newcommand{\eps}{\varepsilon}

\begin{document}
\title[Mean Li-Yorke chaos along some good sequences]
{Mean Li-Yorke chaos along some good sequences}

\author{Jian Li}
\address[J.~Li]{Department of Mathematics, Shantou University, Shantou, Guangdong 515063, P.R. China}
\email{lijian09@mail.ustc.edu.cn}

\author{Yixiao Qiao}
\address[Y.~Qiao]{Department of Mathematics, University of Science and Technology of China,
Hefei, Anhui 230026, P.R. China
-- \and --
Institute of Mathematics, Polish Academy of Sciences, ul. \'Sniadeckich 8, 00-656
Warsaw, Poland}
\email{yxqiao@mail.ustc.edu.cn}
\date{\today}



\begin{abstract}
If a topological dynamical system $(X,T)$ has positive topological entropy,
then it is multivariant mean Li-Yorke chaotic along
a sequence $\{a_k\}_{k=1}^\infty$ of positive integers which
is ``good'' for pointwise ergodic convergence with a mild condition;
more specifically,
there exists a Cantor subset $K$ of $X$ such that
for every $n\ge2$ and pairwise distinct points $x_1,x_2,\dotsc,x_n$ in $K$ we have
\[\liminf_{N\to\infty}\frac{1}{N}\sum_{k=1}^N\max_{1\leq i<j\leq n}
 d(T^{a_k}x_i,T^{a_k}x_j)=0\]
and
\[\limsup_{N\to\infty}\frac{1}{N}\sum_{k=1}^N\min_{1\leq i<j\leq n}
d(T^{a_k}x_i,T^{a_k}x_j)>0.\]
Examples are given for the classic sequences of primes and generalized polynomials.
\end{abstract}

\maketitle

\section{Introduction}
In investigations of topological dynamical systems,
different versions of chaos, which represent complexity in various senses,
attract a lot of attention over the past few decades.
The most studied chaos are Devaney chaos, Li-Yorke chaos, distributional chaos,
positive entropy and weak mixing.
The relationship among them naturally became a central topic as well.
It was known that weak mixing implies Li-Yorke chaos \cite{I}.
In 2002, Huang and Ye \cite{HY02} showed that Devaney chaos implies Li-Yorke chaos.
Later, using ergodic theoretic method,
Blanchard, Glasner, Kolyada and Maass \cite{BGKM} proved that
positive topological entropy also implies Li-Yorke chaos.
A combinatorial proof was given by Kerr and Li \cite{KL07}.

Many of the classical notions in topological dynamical systems
have an analogous version in the mean sense.
Downarowicz \cite{D} observed that mean Li-Yorke chaos
is equivalent to the so-called DC2 chaos
and proved that positive topological entropy implies mean Li-Yorke chaos,
which strengthens the result of \cite{BGKM}.
Huang, Ye and the first author \cite{HLY14} provided a different approach and
showed that positive topological entropy implies a
multivariant version of mean Li-Yorke chaos.
For related topics on chaotic behaviours
we refer to \cite{KL13,HXY15,HJ15} for general group actions,
\cite{GJ} for a new condition implying mean Li-Yorke chaos,
\cite{BH08,FHYZ12,H09} for more careful discussions,
and the survey article \cite{LY16} for more aspects and details.

\medskip

It is natural to study ``sequence versions'' of dynamical notions
(e.g., sequence entropy).
In order to make our statement explicit,
we restate the main theorem of \cite{HLY14} in the following,
which concerns the terminology of multivariant
mean Li-Yorke chaos.
By a \emph{topological dynamical system},
we mean a pair $(X,T)$, where $X$ is a compact metric space
equipped with a metric $d$
and $T$ is a homeomorphism of $X$ onto itself.

\begin{thm}[{\cite[Theorem 1.1]{HLY14}}]\label{thm:result-HLY14}
If a topological dynamical system $(X,T)$ has positive topological entropy,
then it is multivariant mean Li-Yorke chaotic; namely, there exists
a Cantor subset $K$ of $X$ such that
for every integer $n\geq2$ and pairwise distinct points $x_1,x_2,\dotsc,x_n$ in $K$ it holds that
\[\liminf_{N\to\infty}\frac{1}{N}\sum_{k=1}^N\max_{1\leq i<j\leq n} d(T^kx_i,T^kx_j)=0\]
and
\[\limsup_{N\to\infty}\frac{1}{N}\sum_{k=1}^N\min_{1\leq i<j\leq n} d(T^kx_i,T^kx_j)>0.\]
\end{thm}

Motivated by the above consideration and results we naturally ask
if ``sequence chaos'' appears in positive entropy systems;
more precisely, we investigate the following question.

\begin{question}\label{ques}
Let $\{a_k\}_{k=1}^\infty$ be a sequence of positive integers.
Suppose that $(X,T)$ is a topological dynamical system with positive entropy.
Is there a Cantor (in particular, uncountable) subset $K$ of $X$ such that
for every integer $n\geq2$ and pairwise distinct points $x_1,x_2,\dotsc,x_n$ in $K$
it holds that
\[\liminf_{N\to\infty}\frac{1}{N}\sum_{k=1}^N\max_{1\leq i<j\leq n} d(T^{a_k}x_i,T^{a_k}x_j)=0\]
and
\[\limsup_{N\to\infty}\frac{1}{N}\sum_{k=1}^N\min_{1\leq i<j\leq n} d(T^{a_k}x_i,T^{a_k}x_j)>0\;?\]
\end{question}

As it was shown in previous results (e.g., Theorem \ref{thm:result-HLY14}),
Birkhoff's pointwise ergodic theorem
plays a key role. To study Question \ref{ques},
we would follow this idea. However,
for general sequences of positive integers
the situation of pointwise ergodic convergence becomes much more complicated.
Fortunately, there were classic pointwise ergodic theorems for subsequences
of positive integers as well (see Section 3.3 for details), which are quite useful for our argument.
It is worth mentioning that the proofs of these ergodic theorems are extremely nontrivial.
All these facts lead to the following restrictions on sequences.

\medskip

\begin{defn}
An increasing sequence $\{a_k\}_{k=1}^\infty$ of positive integers is
\emph{pointwise good} if for each measure preserving system $(X,\mathcal{B},\mu,T)$
and $f\in L^2(\mu)$,
\[\frac{1}{N}\sum_{k=1}^{N}f(T^{a_k}x)\]
converges for $\mu$-a.e. $x\in X$.
\end{defn}

There are a lot of pointwise good sequences.
In a series of papers, Bourgain \cite{B88a,B88,B89} proved that $\{[p(k)]\}$ is pointwise good provided
$p(x)$ is a polynomial of real coefficients,
where $[\,\cdot\,]$ denotes the integer part of a real number.
This result was generalized to some logarithmico-exponential subpolynomials
\cite{BKQW05}.
The sequence of primes \cite{B88a} and
many return time sequences \cite{BR69,B89,L92} are also pointwise good.
We refer to \cite{J11,PS95} for a more comprehensive understanding of
pointwise ergodic theorems.

If a sequence is pointwise good,
then the ergodic average along this sequence converges almost everywhere.
Nonetheless, we do not have a further understanding of the value of the limit.
Following \cite{L92},
we say that a sequence $\{a_k\}$ is \emph{very good} if for any ergodic system
$(X,\mathcal{B},\mu,T)$ and $f\in L^2(\mu)$, it holds that
\[\lim_{N\to\infty} \frac{1}{N}\sum_{k=1}^{N}f(T^{a_k}x)=\int fd\mu
\quad \text{ for }\, \mu\text{-a.e. }\, x\in X.\]
In this terminology, Birkhoff's pointwise ergodic theorem states
that the sequence of natural numbers is very good.
Following the ideas in \cite{HLY14},
it is routine to check that Theorem \ref{thm:result-HLY14} also holds along any very good sequence;
namely, the answer to Question \ref{ques} is affirmative for every very good sequence.
For example, it was shown in \cite{L92} that the Morse sequence is very good
and in \cite{BKQW05} that sequences generalized by some logarithmico-exponential subpolynomials
(including sequences $\{[k^r]\}$ for all non-integer numbers $r>0$)
are very good.

Unfortunately, there are only few pointwise good sequences which are verified to be very good.
Comparing to the unsolved cases of Question \ref{ques},
this is far from satisfying.
To deal with the remaining cases, we follow the idea in \cite{FW96}
to study the limit of ergodic average along sequences.

\begin{defn}
A sub-$\sigma$-algebra $\mathcal{F}$ of $\mathcal{B}$ is
a \emph{characteristic $\sigma$-algebra for the sequence $\{a_k\}$}
if for every $f\in L^2(\mu)$,
\begin{equation*}
\frac{1}{N}\sum_{k=1}^{N} T^{a_k}f-
\frac{1}{N}\sum_{k=1}^{N} T^{a_k}E(f|\mathcal{F})\to0
\end{equation*}
as $N\to\infty$ in $L^2(\mu)$,
where $E(\cdot|\mathcal{F})$ is the conditional expectation with respect to $\mathcal{F}$.
\end{defn}

As we will show in Theorem \ref{thm-factor-Pinkser},
the Pinsker $\sigma$-algebra of a system is characteristic for
all sequences $\{a_k\}$ satisfying Condition ($*$) below,
which helps us overcome additional difficulties in the proof of our main theorem.

\medskip

\noindent\textbf{Condition ($*$).}
\,For every $L>0$,
\[\lim_{N\to\infty}\frac{1}{N^2}\#\{(i,j)\in[1,N]^2:\,|a_i-a_j|\leq L\}=0.\]

\medskip

We note here that Condition ($*$) is quite mild.
In fact, it is not hard to check that any strictly increasing sequence of positive integers
satisfies Condition ($*$) (see Lemma \ref{lem:sequence-condition-star}),
and that all those pointwise good sequences we previously mentioned
satisfy Condition ($*$) as well
(see Remark \ref{rem:good-sequences-examples}).
Moreover, for any pointwise good sequence satisfying Condition ($*$),
the limit of ergodic average is a constant on each atom of the generating partition of
the Pinsker $\sigma$-algebra.

\medskip

We now state the main result of this paper in the following,
solving Question \ref{ques} affirmatively for a large class of sequences.

\begin{thm}[Main theorem]\label{thm-main-result}
Let $\{a_k\}$ be a pointwise good sequence satisfying Condition ($*$).
Suppose that $(X,T)$ is a topological dynamical system with positive topological entropy.
Then $(X,T)$ is multivariant mean Li-Yorke chaotic along the sequence $\{a_k\}$;
namely, there exists a Cantor subset $K$ of $X$ such that
for every integer $n\geq2$ and pairwise distinct points $x_1,x_2,\dotsc,x_n$ in $K$
we have
\[\liminf_{N\to\infty}\frac{1}{N}\sum_{k=1}^N\max_{1\leq i<j\leq n} d(T^{a_k}x_i,T^{a_k}x_j)=0\]
and
\[\limsup_{N\to\infty}\frac{1}{N}\sum_{k=1}^N\min_{1\leq i<j\leq n} d(T^{a_k}x_i,T^{a_k}x_j)>0.\]
\end{thm}

\medskip

The proof of Theorem \ref{thm-main-result} is based on the techniques used in \cite{HLY14}.
A main new ingredient in this paper is the convergence argument for sequences given in Section 3.

\begin{cor}
Let $\{a_k\}$ be a pointwise good sequence satisfying Condition ($*$).
Suppose that $(X,T)$ is a topological dynamical system with positive topological entropy.
Then $(X,T)$ is Li-Yorke chaotic along the sequence $\{a_k\}$;
namely, there exists a Cantor subset $K$ of $X$ such that
for any $x,y\in K$ with $x\neq y$
we have
\[\liminf_{N\to\infty}d(T^{a_k}x,T^{a_k}y)=0
\,\;\,\text{ and }\,\;\,
\limsup_{N\to\infty}d(T^{a_k}x,T^{a_k}y)>0.\]
\end{cor}

\medskip

As we mentioned previously, the prime sequence and $\{[p(k)]\}$ (where $p(x)$ is a polynomial
of real coefficients)
are pointwise good and satisfy Condition ($*$).

\begin{cor}
Any positive entropy system is multivariant mean Li-Yorke chaotic along the prime sequence.
\end{cor}

\begin{cor}
Any positive entropy system is multivariant mean Li-Yorke chaotic along the sequence $\{[p(k)]\}$,
where $p(x)$ is a polynomial of real coefficients.
\end{cor}

\medskip

This paper is organized as follows. In Section 2, we review some necessary notions and required properties.
In Section 3, we study ergodic average along sequences.
More precisely, we show that
the Pinsker $\sigma$-algebra is a characteristic $\sigma$-algebra for any sequence
satisfying Condition ($*$)
(Theorem \ref{thm-factor-Pinkser}).
Moreover, for any pointwise good sequence satisfying Condition ($*$)
we have a decomposition for any invariant measure
(Theorem \ref{thm:tau-x}); meanwhile,
the limit of ergodic average is a constant on each atom of the generating partition
of the Pinsker $\sigma$-algebra (Theorem \ref{thm:Pinsker-constant-atom}).
In Section 4,
we prove Theorem \ref{thm-main-result},
which holds for all continuous surjective maps as well
(Theorem \ref{thm:non-invertible-case}).

\medskip

\section{Preliminaries}

For convenience,
our notations will be as close to \cite{HLY14} as possible. For the reader who is not familiar with notions in ergodic theory and dynamical systems, we refer to \cite{EW,PW}.

\subsection{Mycielski's Theorem}
Let $X$ be a compact metric space and $C(X)$ the space of all real-valued continuous functions on $X$
equipped with the supremum norm.
For $n\geq2$, denote $X^n=X\times X\times\dotsb\times X$ ($n$-copies).
Set $\Delta_n=\{(x,x,\dotsc,x)\in X^n\colon x\in X\}$
and $\Delta^{(n)}=\{(x_1,x_2,\dotsc,x_n)\in X^n:\,\text{ there exists }1\leq i<j\leq n\text{ such that }x_i=x_j\}$.
We shall use the following version of Mycielski's theorem.

\begin{thm}[Cf. {\cite[Theorem 1]{My}}]\label{Myc}
Assume that $X$ is a perfect compact metric space.
If for every integer $n\geq2$, $R_n$ is a dense $G_\delta$ subset of $X^n$,
then there exists a dense subset $K$ of $X$ which is a union of countably many Cantor sets
such that
$K^n\subset R_n\cup\Delta^{(n)}$ holds for all integers $n\geq2$.
\end{thm}

\subsection{Conditional expectation and disintegration}\label{Conditional expectation}
Let $(X,\mathcal{B},\mu)$ be a probability space and
$\mathcal{A}$ a sub-$\sigma$-algebra of $\mathcal{B}$.
The conditional expectation is a map
\[E(\,\cdot\,|\mathcal{A})\colon L^1(X,\mathcal{B},\mu)\to L^1(X,\mathcal{A},\mu)\]
satisfying the following conditions:
\begin{enumerate}
  \item for every $f\in L^1(X,\mathcal{B},\mu)$, $E(f|\mathcal{A})$ is  $\mathcal{A}$-measurable;
  \item if $g$ is $\mathcal{A}$-measurable and $fg\in L^1(X,\mathcal{B},\mu)$,
  then $E(fg|\mathcal{A})=gE(f|\mathcal{A})$;
  \item if $f\in L^p(X,\mathcal{B},\mu)$ for some $p\geq1$,
  then $E(f|\mathcal{A})\in L^p(X,\mathcal{A},\mu)$ and
  \[\Vert E(f|\mathcal{A})\Vert_{L^p}\leq \Vert f \Vert_{L^p}.\]
\end{enumerate}

The Martingale theorem is well known
(see e.g., \cite[Theorem 14.26]{Gl1}, \cite[Chapter 5.2]{EW}).
\begin{thm}[Martingale theorem]\label{thm:Margingale-theorem}
Let $(X,\mathcal{B},\mu)$ be a probability space.
Suppose that $\{\mathcal{A}_n\}_{n=1}^\infty$ is a decreasing sequence (resp.
an increasing sequence)
of sub-$\sigma$-algebras of $\mathcal{B}$ and
$\mathcal{A}=\bigcap_{n\geq 1}\mathcal{A}_n$ (resp.
$\mathcal{A}=\bigvee_{n\geq 1}\mathcal{A}_n$).
Then for any $f\in L^1(\mu)$,
\[E(f|\mathcal{A}_n)\to E(f|\mathcal{A})\]
as $n\to\infty$ in $L^1(\mu)$ and $\mu$-almost everywhere.
\end{thm}

Let $X$ be a compact metric space.
Denote by $\mathcal{B}_X$ the Borel $\sigma$-algebra of $X$
and $\mathcal{M}(X)$ the set of all Borel probability measures on $X$.
For $\mu\in\mathcal{M}(X)$,
let $\mathcal{B}_\mu$ be the completion of $\mathcal{B}_X$ under the measure $\mu$.
Then $(X,\mathcal{B}_\mu,\mu)$ is a Lebesgue space.
A finite partition of $X$ is a finite family of pairwise distinct measurable subsets of $X$
whose union is $X$.
If $\{ \alpha_i\}_{i\in I}$ is a countable family of finite
partitions of $X$, then we say that $\alpha=\bigvee_{i\in I}\alpha_i$
is a {\it measurable partition}. The sets $A\in
\mathcal{B}_\mu$, which are unions of atoms of $\alpha$, form a
sub-$\sigma$-algebra $\mathcal{B}_\mu$, which we denote by $\widehat{\alpha}$
or $\alpha$ if there is no ambiguity. Every sub-$\sigma$-algebra of
$\mathcal{B}_\mu$ coincides with a $\sigma$-algebra constructed in
this way (mod $\mu$).

Let $\mathcal{F}$ be a sub-$\sigma$-algebra of $\mathcal{B}_\mu$ and
$\alpha$ a measurable partition of $X$ with $\widehat{\alpha}=\mathcal{F}$ (mod $\mu$).
Then $\mu$ can be disintegrated over $\mathcal{F}$ as
\[\mu=\int_X \mu_x d \mu(x),\]
where $\mu_x\in \mathcal{M}(X)$ and $\mu_x(\alpha(x))=1$ for $\mu$-a.e. $x\in X$.
The disintegration can be characterized by the
properties \eqref{meas1} and \eqref{meas3} as follows:
\begin{align}
&\text{for every } f \in L^1(X,\mathcal{B}_X,\mu),\
f \in L^1(X,\mathcal{B}_X,\mu_x)\ \text{for $\mu$-a.e. } x\in X, \label{meas1}\\
& \text{and the map $x \mapsto \int_X  f(y)\,d\mu_x(y)$ is in $L^1(X,\mathcal{F},\mu)$}; \notag\\
& \text{for every } f\in L^1(X,\mathcal{B}_X,\mu),\
E(f|\mathcal{F})(x)=\int_X f\,d\mu_{x}\ \text{for $\mu$-a.e. } x\in X. \label{meas3}
\end{align}
Then for any $f \in L^1(X,\mathcal{B}_X,\mu)$, one has
\begin{equation*}
\int_X \left(\int_X f(y)\,d\mu_x(y) \right)\, d\mu(x)=\int_X f \,d\mu.
\end{equation*}

\subsection{Ergodic theory}
By a \emph{measure-preserving system}, we mean a quadruple $(X,\mathcal{B},\mu,T)$,
where $(X,\mathcal{B},\mu)$ is a probability space and
$T\colon X\to X$ is an invertible measure-preserving transformation.
A measure-preserving system $(X,\mathcal{B},\mu,T)$ is called \emph{ergodic} if
the only members $B\in\mathcal{B}$ with $T^{-1}B=B$ satisfy $\mu(B)=0$ or $\mu(B)=1$.

Let $\alpha$ be a finite partition of $X$.
The measure-theoretic entropy of $\mu$ relative to $\alpha$ is denoted by $h_\mu(T,\alpha)$,
and the measure-theoretic entropy of $\mu$ is defined as
\[h_\mu(T)=\sup_{\alpha} h_\mu(T,\alpha),\]
where the supremum ranges over all finite partitions of $X$.

The \emph{Pinsker $\sigma$-algebra} of a system $(X,\mathcal{B}, \mu,T)$ is defined as
\[P_\mu(T)=\{A\in \mathcal{B}\colon h_\mu(T,\{A, X\setminus A\})=0\}.\]
It is easy to see that $P_\mu(T)$ is $T$-invariant.
The Rokhlin-Sinai theorem identifies
the Pinsker $\sigma$-algebra as the ``remote past'' of a generating partition.

\begin{thm}[\cite{RS61}]
For a measure-preserving system $(X,\mathcal{B},\mu,T)$,
there exists a sub-$\sigma$-algebra $\mathcal{P}$ of $\mathcal{B}$
such that $T^{-1}\mathcal{P}\subset\mathcal{P}$, $\bigvee_{k=0}^\infty T^k\mathcal{P}=\mathcal{B}$
and $\bigcap_{n=0}^\infty T^{-k}\mathcal{P}=P_\mu(T)$.
\end{thm}

\subsection{Topological dynamics}
Let $(X,T)$ be a topological dynamical system. For a point $x\in X$,
the \emph{stable set} of $x$ is defined as
\[W^{s}(X,T)=\left\{y\in X\colon\lim_{k\to\infty}d(T^kx,T^ky)=0\right\},\]
and the \emph{unstable set} of $x$ is defined as
\[W^{u}(X,T)=\left\{y\in X\colon\lim_{k\to\infty}d(T^{-k}x,T^{-k}y)=0\right\}.\]

Let $\mathcal{M}(X,T)$ (resp. $\mathcal{M}^e(X,T)$) be the collection of all $T$-invariant (resp.
ergodic $T$-invariant) Borel probability measures on $X$.
For every $\mu\in\mathcal{M}(X,T)$, $(X,\mathcal{B}_X,\mu,T)$ is a measure-preserving system.
The variational principle is famous:
\[h_\text{top}(X,T)=\sup_{\mu\in\mathcal{M}(X,T)}h_\mu(X,T)=\sup_{\mu\in\mathcal{M}^e(X,T)}h_\mu(X,T),\]
where $h_\text{top}(X,T)$ denotes the topological entropy of $(X,T)$.

\medskip

\section{The Pinsker $\sigma$-algebra and pointwise good sequences}
\subsection{Characteristic $\sigma$-algebras}
\begin{lem}\label{lem:sequence-condition-star}
Any strictly increasing sequence of positive integers
satisfies Condition ($*$).
\end{lem}
\begin{proof}
Fix a strictly increasing sequence $\{a_k\}$ of positive integers and $L>0$.
It is clear that $|a_i-a_j|\leq L$ implies $|i-j|\leq L$.
For every integer $N>L$, we have
$$
\#\{(i,j)\in[1,N]^2\colon |a_{i}-a_{j}|\leq L\}
\leq\# \{(i,j)\in [1,N]^{2}\colon |i-j|\leq L\}
\leq 2NL.
$$
Thus,
\[\lim_{N\to\infty}\frac{1}{N^2}\#\{(i,j)\in[1,N]^2\colon |a_i-a_j|\leq L\}=0.\]
\end{proof}

\begin{thm}\label{thm-factor-Pinkser}
Suppose that $\{a_k\}$ is a sequence of positive integers satisfying Condition ($*$)
and $(X,\mathcal{B},\mu,T)$ is a measure-preserving system.
Then the Pinsker $\sigma$-algebra $P_\mu(T)$ is a
characteristic $\sigma$-algebra for the sequence $\{a_k\}$.
\end{thm}
\begin{proof}
Fix $f\in L^2(\mu)$ and $\eps>0$.
Choose a partition $\mathcal{P}$ as in the Rohlin-Sinai theorem.
By Theorem \ref{thm:Margingale-theorem},
there exist $m>0$ and $g_m=E(f|T^m\mathcal{P})$ such that
\[\Vert g_m-f\Vert_{L^2(\mu)}<\eps.\]
Thus, it holds that
\begin{align*}
\biggl\Vert \frac{1}{N}\sum_{k=1}^{N} (T^{a_k}f- T^{a_k}g_m)\biggr\Vert_{L^2(\mu)}
&\leq \frac{1}{N} \sum_{k=1}^{N} \biggl\Vert T^{a_k}f- T^{a_k}g_m\biggr\Vert_{L^2(\mu)}\\
   &=\frac{1}{N} \sum_{k=1}^{N} \Vert f-g_m\Vert_{L^2(\mu)}<\eps
\end{align*}
and
\begin{align*}
&\biggl\Vert \frac{1}{N}\sum_{k=1}^{N} \bigl(T^{a_k}E(g_m|P_\mu(T))-
T^{a_k}E(f|P_\mu(T))\bigr)\biggr\Vert_{L^2(\mu)}\\
 &\qquad \leq\frac{1}{N} \sum_{k=1}^{N} \Vert T^{a_k}E(g_m|P_\mu(T))- T^{a_k}E(f|P_\mu(T))\Vert_{L^2(\mu)}\\
 &\qquad= \frac{1}{N} \sum_{k=1}^{N} \Vert E(g_m|P_\mu(T))-E(f|P_\mu(T))\Vert_{L^2(\mu)}\\
 &\qquad= \frac{1}{N} \sum_{k=1}^{N}\Vert E(g_m-f|P_\mu(T))\Vert_{L^2(\mu)}\\
 &\qquad\leq \frac{1}{N} \sum_{k=1}^{N}\Vert g_m-f\Vert_{L^2(\mu)} <\eps.
\end{align*}
By Theorem \ref{thm:Margingale-theorem} again,
there exist $n>0$ and $h_n= E(g_m|T^{-n}\mathcal{P})$ such that
\[\Vert h_n-E(g_m|P_\mu(T))\Vert_{L^2(\mu)}<\eps.\]
It follows that
\begin{align*}
\biggl\Vert \frac{1}{N}\sum_{k=1}^{N} \bigl(T^{a_k}h_n-T^{a_k}E(g_m|P_\mu(T))\bigr)
\biggr\Vert_{L^2(\mu)}<\eps.
\end{align*}

Notice that
\begin{align*}
\biggl\Vert \frac{1}{N}\sum_{k=1}^{N}(T^{a_k}g_m-T^{a_k}h_n)
\biggr\Vert_{L^2(\mu)}^2
&=\int \Bigl|\frac{1}{N}\sum_{k=1}^{N} (T^{a_k}g_m-T^{a_k}h_n)\Bigr|^2 d\mu\\
&= \frac{1}{N^2} \sum_{i,j=1}^{N} \int\bigl(T^{a_i}g_m-T^{a_i}h_n\bigr)
\bigl(T^{a_j}g_m-T^{a_j}h_n\bigr) d\mu \\
& =\frac{1}{N^2}\sum_{i,j=1}^{N} \bigl(A_{ij}-B_{ij}+C_{ij}-D_{ij}\bigr),
\end{align*}
where
\begin{align*}
A_{ij}&=\int T^{a_i}g_m \cdot T^{a_j}g_m d\mu,\quad
B_{ij}=\int T^{a_i}h_n\cdot T^{a_j}g_m d\mu,\\
C_{ij}&=\int T^{a_i}h_n\cdot T^{a_j}h_nd\mu,\quad
D_{ij}=\int T^{a_i}g_m\cdot T^{a_j}h_nd\mu.
\end{align*}

\begin{Claim1}
If $a_j-m\geq n+a_i$, then we have $A_{ij}=B_{ij}$ and $C_{ij}=D_{ij}$.
\end{Claim1}
\begin{proof}[Proof of Claim 1]
First, we have
\begin{align*}
B_{ij}
=&\int h_n \circ T^{a_i}\cdot  g_m\circ T^{a_j} d\mu
=\int \Bigl(h_n\cdot g_m\circ T^{a_j-a_i}\Bigr) \circ T^{a_i} d\mu \\
=&\int h_n\cdot g_m\circ T^{a_j-a_i}d\mu.
\end{align*}
Recall that $g_m=E(f|T^m\mathcal{P})$.
Then $g_m$ is $T^m\mathcal{P}$-measurable
and hence $g_m\circ T^{a_j-a_i}$ is $T^{-(a_j-a_i-m)}\mathcal{P}$-measurable.
As $a_j-m\geq n+a_i$, $a_j-a_i-m\geq n$,
we know that $g_m\circ T^{a_j-a_i}$ is $T^{-n}\mathcal{P}$-measurable.
Recall that $h_n=E(g_m|T^{-n}\mathcal{P})$.
By (2) of Section \ref{Conditional expectation}, we have
\begin{align*}
B_{ij} =&\int E(g_m\cdot g_m \circ T^{a_j-a_i}|T^{-n}\mathcal{P}) d\mu
=\int g_m\cdot g_m \circ T^{a_j-a_i} d\mu\\
=&\int \bigl(g_m\cdot g_m \circ T^{a_j-a_i}\bigr)\circ T^{a_i} d\mu = A_{ij}.
\end{align*}
Now we consider the term $C_{ij}$:
\begin{align*}
C_{ij}&=\int h_n\circ T^{a_i}\cdot h_n\circ T^{a_j}d\mu=\int h_n \cdot h_n\circ T^{a_j-a_i} d\mu.
\end{align*}
As $a_j-m\geq n+a_i$, $a_j\geq a_i$,
we have that $E(g_m|T^{-n}\mathcal{P})\circ T^{a_j-a_i}$ is $T^{-n}\mathcal{P}$-measurable.
Recall that $h_n=E(g_m|T^{-n}\mathcal{P})$.
By (2) of Section \ref{Conditional expectation}, we have
\begin{align*}
C_{ij}=&\int E\Bigl(g_m \cdot h_n\circ T^{a_j-a_i}|T^{-n}\mathcal{P}\Bigr)d\mu
=\int g_m \cdot h_n\circ T^{a_j-a_i}d\mu\\
=&\int g_m\circ T^{a_i}\cdot h_n\circ T^{a_j} d\mu=D_{ij}.
\end{align*}
This ends the proof of Claim 1.
\end{proof}

Similarly, we have the following

\begin{Claim2}
If $a_i-m\geq n+a_j$, then we have $A_{ij}=D_{ij}$ and $B_{ij}=C_{ij}$.
\end{Claim2}

By H\"older's inequality, it is easy to see that
\[\max\Bigl\{\bigl|A_{ij}\bigr|, \bigl|B_{ij}\bigr|, \bigl|C_{ij}\bigr|, \bigl|D_{ij}\bigr|\Bigr\}
\leq \Vert g_m\Vert_{L^2(\mu)}^2.\]
Note that the sequence $\{a_k\}$ satisfies Condition ($*$).
For $n+m-1$ there exists $N_0>0$ such that whenever $N\geq N_0$ it holds that
\[\frac{1}{N^2}\#\{(i,j)\in [1,N]^2\colon |a_j-a_i|\leq n+m-1\}<
\frac{\eps^2}{4\Vert g_m\Vert_{L^2(\mu)} ^2}.\]
Therefore, when $N\geq N_0$ we have
\begin{align*}
&\biggl \Vert \frac{1}{N}\sum_{k=1}^{N} (T^{a_k}g_m-T^{a_k}h_n)
\biggr\Vert_{L^2(\mu)}^2\\
&\qquad =\frac{1}{N^2}\sum_{i,j=1}^{N} \bigl(A_{ij}-B_{ij}+C_{ij}-D_{ij}\bigr)
= \frac{1}{N^2} \sum_{\substack{i,j=1\\|a_j-a_i| \leq n+m-1}}^{N}
\bigl(A_{ij}-B_{ij}+C_{ij}-D_{ij}\bigr)\\
&\qquad\leq 4\Vert g_m\Vert_{L^2(\mu)}^2\cdot
\frac{1}{N^2}\#\{(i,j)\in [1,N]^2\colon |a_j-a_i|\leq n+m-1\}
< \eps^2.
\end{align*}
To sum up, for $N\geq N_0$ we have
\begin{align*}
&\biggl\Vert \frac{1}{N}\sum_{k=1}^{N}
\Bigl(T^{a_k}f- T^{a_k}E(f|P_\mu(T))\Bigr)\biggr\Vert_{L^2(\mu)}\\
&\qquad\leq \biggl\Vert \frac{1}{N}\sum_{k=1}^{N} (T^{a_k}f- T^{a_k}g_m)\biggr\Vert_{L^2(\mu)}+
\biggl \Vert \frac{1}{N}\sum_{k=2}^{N} (T^{a_k}g_m-T^{a_k}h_n)
\biggr\Vert_{L^2(\mu)}\\
&\qquad\qquad +\biggl\Vert \frac{1}{N}\sum_{k=1}^{N} \bigl(T^{a_k}h_n-
T^{a_k}E(g_m|P_\mu(T))\bigr)\biggr\Vert_{L^2(\mu)}\\
&\qquad\qquad +
\biggl\Vert \frac{1}{N}\sum_{k=1}^{N} \bigl(T^{a_k}E(g_m|P_\mu(T))-
T^{a_k}E(f|P_\mu(T))\bigr)\biggr\Vert_{L^2(\mu)}\\
&\qquad< \eps+\eps+\eps+\eps=4\eps.
\end{align*}
This ends the proof.
\end{proof}

\begin{rem}
We also refer to \cite{DL96} for a variant of Theorem \ref{thm-factor-Pinkser} which deals with polynomial iterates and pointwise convergence.
\end{rem}

\subsection{Good sequences for pointwise convergence}

Similar to the ergodic decomposition theorem,
we have the following decomposition of an invariant measure with respect to any pointwise good sequence.

\begin{thm}\label{thm:tau-x}
Let $\{a_k\}$ be a pointwise good sequence of positive integers.
Suppose that $(X,T)$ is a topological dynamical system and $\mu\in M(X,T)$.
Then there exists a disintegration of $\mu$,
\[\mu=\int \tau_xd\mu(x),\]
in the sense that there exists a Borel subset $X_0$ of $X$ with $\mu(X_0)=1$ such that
for any $x\in X_0$ and $f\in C(X)$, it holds that
\[\lim_{N\to\infty}\frac{1}{N}\sum_{k=1}^{N}f(T^{a_k}x)=\int fd\tau_x\]
and
\[\int \int fd\tau_xd\mu(x)=\int fd\mu.\]
\end{thm}
\begin{proof}
Fix a countable dense subset $\{f_n\}$ of
$C(X)$.
As $\{a_k\}$ is pointwise good, there exists a Borel subset $X_0$ of $X$ with $\mu(X_0)=1$
such that for every $x\in X_0$ and $f_n$, we have
\[\lim_{N\to\infty}\frac{1}{N}\sum_{k=1}^{N}f_n(T^{a_k}x)=\overline{f}_n(x).\]
Fix $x\in X_0$. Define
\begin{align*}
L_x\colon C(X)\to \mathbb{R}, \;\; f\mapsto \lim_{N\to\infty}\frac{1}{N}\sum_{k=1}^{N}f(T^{a_k}x).
\end{align*}
Since $\{f_n\}$ is dense in $C(X)$ endowed with the supremum norm, $L_x$ is well defined.
Moreover, $L_x$ is a positive linear functional with $L_x(1)=1$.
By the Riesz Representation Theorem, there exists $\tau_x\in M(X)$ such that for any $f\in C(X)$,
\[\lim_{N\to\infty}\frac{1}{N}\sum_{k=1}^{N}f(T^{a_k}x)=\int fd\tau_x.\]
By Lebesgue's Dominated Convergence Theorem,
\begin{align*}
\int \left(\int fd\tau_x\right)d\mu(x)&=
    \int \left(\lim_{N\to\infty}\frac{1}{N}\sum_{k=1}^{N}f(T^{a_k}x)\right)d\mu(x)\\
&= \lim_{N\to\infty}\frac{1}{N}\sum_{k=1}^{N} \int f(T^{a_k}x)d\mu(x)=\int fd\mu.
\end{align*}
\end{proof}

\begin{thm} \label{thm:Pinsker-constant-atom}
Let $\{a_k\}$ be a pointwise good sequence,
$(X,T)$ a topological dynamical system and $\mu\in M(X,T)$.
If the Pinsker $\sigma$-algebra $P_\mu(T)$ is a
characteristic $\sigma$-algebra for the sequence $\{a_k\}$, then
there exists a Borel subset $X_1$ of $X$ with $\mu(X_1)=1$ such that
for any $x\in X_1$ and $f\in C(X)$, it holds that
\[\int fd\tau_x=\int\left(\int fd\tau_y\right)d\mu_x(y),\]
where
\[\mu=\int \tau_xd\mu(x)\quad \text{ and } \quad \mu=\int \mu_xd\mu(x)\]
are the disintegrations of $\mu$ as in Theorem \ref{thm:tau-x}
and over $P_\mu(T)$ respectively.
\end{thm}

\begin{proof}
As $\{a_k\}$ is pointwise good,
for every $f\in C(X)$ there exist $\bar f$ and $f^*$ in $L^2(\mu)$ such that
\[\frac{1}{N}\sum_{k=1}^N T^{a_k} f\to \bar f
\;\;\text{and}\;\; \frac{1}{N}\sum_{k=1}^N T^{a_k}E(f|P_\mu(T))\to f^*\]
$\mu$-almost everywhere.
As the Pinsker $\sigma$-algebra $P_\mu(T)$ is a
characteristic $\sigma$-algebra for the sequence $\{a_k\}$,
by the definition we have
\begin{equation*}
\frac{1}{N}\sum_{k=1}^{N} T^{a_k}f-
\frac{1}{N}\sum_{k=1}^{N} T^{a_k}E(f|P_\mu(T))\to 0
\end{equation*}
as $N\to\infty$ in $L^2(\mu)$.
So there exists a strictly increasing sequence $\{N_i\}$ of positive integers such that
\[
\frac{1}{N_i}\sum_{k=1}^{N_i} T^{a_k}f-
\frac{1}{N_i}\sum_{k=1}^{N_i} T^{a_k}E(f|P_\mu(T))\to0
\]
as $i\to\infty$ $\mu$-almost everywhere. Hence $\bar f(x)= f^*(x)$ for $\mu$-a.e. $x\in X$.
Clearly, $f^*$ is $P_\mu(T)$-measurable, so is $\bar f$.
Let
\[\mu=\int \tau_xd\mu(x)\]
be the disintegration of $\mu$
as in Theorem \ref{thm:tau-x}.
Then $\bar f(x)=\int f d\tau_x$ for $\mu$-a.e. $x\in X$.
Let
\[\mu=\int \mu_xd\mu(x)\]
be the disintegration of $\mu$ over the Pinsker $\sigma$-algebra $P_\mu(T)$.
As $\bar f$ is $P_\mu(T)$-measurable, by \eqref{meas3}
we have $\bar f(x)=\int \bar f d\mu_x$ for $\mu$-a.e. $x\in X$.

Now fix a countable dense subset $\{f_n\}$ of $C(X)$.
By the above discussion, there exists a Borel subset $X_1$ of $X_0$ with $\mu(X_1)=1$
such that for every $x\in X_1$ and $n\in\mathbb{N}$,
\begin{equation}\label{eq:fn-taux}
\int f_n d\tau_x = \int\left(\int f_n d\tau_y\right)d\mu_x(y).
\end{equation}
As any function in $C(X)$ can be uniformly approximated by $\{f_n\}$,
\eqref{eq:fn-taux} holds for all $f\in C(X)$.
\end{proof}

\begin{rem}
Let $\alpha$ be a measurable partition generating $P_\mu(T)$.
Then $\mu_x(\alpha(x))=1$ for $\mu$-a.e. $x\in X$, where $\alpha(x)$ is the atom of $\alpha$ containing $x$.
If a function $f$ is $P_\mu(T)$-measurable, then we know from \eqref{meas3} that
for $\mu$-a.e. $x\in X$,
$f$ is a constant almost everywhere with respect to $\mu_x$ on $\alpha(x)$.
By Theorem \ref{thm:Pinsker-constant-atom},
if a sequence is pointwise good and satisfies Condition ($*$), then
the limit of the ergodic average along this sequence
is a constant on any atom of the generating partition
of $P_\mu(T)$.
\end{rem}

If $(X,\mathcal{B},\mu,T)$ is a Kolmogorov system, then the Pinsker factor is trivial.
This yields the following.
\begin{cor}
Let $\{a_k\}$ be a pointwise good sequence satisfying Condition ($*$) and $(X,\mathcal{B},\mu,T)$ a Kolmogorov system.
Then for every $f\in L^2(\mu)$,
\begin{equation*}
\lim_{N\to\infty}\frac{1}{N}\sum_{k=1}^{N}f(T^{a_k}x) = \int fd\mu
\end{equation*}
holds for $\mu$-a.e. $x\in X$.
\end{cor}

\subsection{Examples of pointwise good sequences}
There are many results on pointwise good/bad sequences,
we refer the reader to \cite{J11,PS95} for recent work on this topic.
Here let us list some sequences of positive integers
which are both pointwise good and satisfying Condition ($*$).

\begin{thm}[{\cite[Theorem 1]{B88a}}]
The sequence of prime numbers is pointwise good.
\end{thm}

\begin{thm}[{\cite[Theorem 2]{B89}}]\label{pol}
For any polynomial $$p(x)=b_{0}+b_{1}x+\cdots+b_{m}x^{m},$$
where $m\ge1$, $b_0,\cdots,b_m\in\mathbb{R}$, and $b_m>0$,
the sequence $\{[p(k)]\}$ is pointwise good.
\end{thm}

\begin{rem}
There are very good random sequences of integers that are not too sparse,
including ones that grow faster than any polynomial (see \cite[Proposition~8.2]{B88}).
\end{rem}

Furthermore, a lot of sequences generated by logarithmico-exponential subpolynomials are pointwise good
(for details see \cite[Theorems 3.4, 3.5, 3.8]{BKQW05}).
In particular, one has the following:
\begin{thm}\label{thm:BKQW05}
For every non-integer number $r>0$, the sequence $\{[k^r]\}$ is very good.
\end{thm}

\begin{thm}[{\cite[Theorem in Appendix]{B89}}]
Let $(X,\mathcal{B},\mu,T)$ be an ergodic measure-preserving system
and $A\in\mathcal{B}$ with $\mu(A)>0$.
Then for $\mu$-a.e. $x\in X$, the sequence $\{n\in\mathbb{N}\colon T^nx\in A\}$ is pointwise good.
\end{thm}

\begin{thm}[{\cite[Theorem 1]{BR69}}]
Let $(X,T)$ be a minimal equicontinuous system with
the unique invariant measure $\mu$.
For every $x\in X$ and
every subset $A$ of $X$ with $\mu(A)>0$ and $\mu(\partial A)=0$
(where $\partial A$ is the boundary of $A$),
the sequence $\{n\in\mathbb{N}\colon T^nx\in A\}$ is pointwise good.
\end{thm}

\begin{thm}[{\cite[Theorem 3]{L92}}]
Return time sequences for dynamical systems which are abelian extensions of
translations on compact abelian groups are pointwise good.
Moreover, the Morse sequence is very good.
\end{thm}

\begin{rem}\label{rem:good-sequences-examples}
Similar to the proof of Lemma \ref{lem:sequence-condition-star},
it is not hard to check that the sequences
in Thereoms \ref{pol} and \ref{thm:BKQW05}
satisfy Condition ($*$).
\end{rem}

\medskip

\section{Proof of the main theorem}
\subsection{Dynamical systems with positive entropy}
Let $(X,T)$ be a topological dynamical system, $\mu\in\mathcal{M}(X,T)$,
and $\mathcal{B}_{\mu}$ the completion of $\mathcal{B}_X$ under $\mu$.
Then $(X,\mathcal{B}_\mu,\mu,T)$ is a Lebesgue system.
Let $P_\mu(T)$ be the Pinsker $\sigma$-algebra of $(X,\mathcal{B}_\mu,\mu,T)$
and \[\mu=\int_X \mu_x d \mu(x)\]
the disintegration of $\mu$ over $P_\mu(T)$.
Under the above settings, we first state the following lemma.

\begin{lem}[{\cite[Lemma 3.1]{HLY14}}] \label{lem:HYL14}
If $\mu\in\mathcal{M}^e(X,T)$ and $h_\mu(T)>0$,
then for $\mu$-a.e. $x\in X$,
\begin{align*}
\overline{W^s(x,T)\cap \supp(\mu_x)}=\supp(\mu_x) \text{ and }
\overline{W^u(x,T)\cap\supp(\mu_x)}=\supp(\mu_x).
\end{align*}
\end{lem}

\medskip

By Theorem \ref{thm-factor-Pinkser} and the variational principle,
our main result (Theorem~\ref{thm-main-result}) follows from
Theorem \ref{thm-1}:

\begin{thm}\label{thm-1}
Let $(X,T)$ be a topological dynamical system and $\mu\in\mathcal{M}^e(X,T)$ with $h_\mu(T)>0$.
If $\{a_k\}$ is a pointwise good sequence and
the Pinsker $\sigma$-algebra $P_\mu(T)$ is a characteristic $\sigma$-algebra for the sequence $\{a_k\}$,
then for $\mu$-a.e. $x\in X$, there exists a Cantor
subset $K_x\subset\overline{W^s(x,T)}\cap\overline{W^u(x,T)}$
satisfying that
for every integer $n\geq2$ and pairwise distinct points $x_1,x_2,\dotsc,x_n$ in $K_x$
we have
\[\liminf_{N\to\infty}\frac{1}{N}\sum_{k=1}^N\max_{1\leq i<j\leq n} d(T^{a_k}x_i,T^{a_k}x_j)=0\]
and
\[\limsup_{N\to\infty}\frac{1}{N}\sum_{k=1}^N\min_{1\leq i<j\leq n} d(T^{a_k}x_i,T^{a_k}x_j)
\geq \eta_{x,n},\]
where $\eta_{x,n}$ is a positive constant depending only on $x$ and $n$.
\end{thm}

\begin{proof}
Let \[\mu=\int_X\mu_xd\mu(x)\] be
the disintegration of $\mu$ over $P_\mu(T)$.
By Lemma~\ref{lem:HYL14}, there exists a Borel subset $X_1$ of $X$ with $\mu(X_1)=1$
such that for any $x\in X_1$,
\begin{align*}
\overline{W^s(x,T)\cap \supp(\mu_x)}=\supp(\mu_x) \,\text{ and }\,
\overline{W^u(x,T)\cap\supp(\mu_x)}=\supp(\mu_x).
\end{align*}
Fix an integer $n\geq 2$.
Set
\[MP_n(X,T)=
\Bigl\{(x_1,x_2,\dotsc,x_n)\in X^n\colon
\liminf_{N\to\infty}\frac{1}{N}\sum_{k=1}^N\max_{1\leq i<j\leq n} d(T^{a_k}x_i,T^{a_k}x_j)=0\Bigr\}.\]
By a similar argument as in \cite[Lemma 2.4]{HLY14}, $MP_n(X,T)$ is a $G_\delta$ subset of $X^n$.
It is clear that $(x_1,x_2,\dotsc,x_n)\in MP_n(X,T)$ for any $x_1,x_2,\dotsc,x_n$ in $W^s(x,T)$.
Thus, for each $x\in X_1$, $MP_n(X,T)\cap\supp(\mu_x)^n$ is a dense $G_\delta$ subset of $\supp(\mu_x)^n$.

Define a measure $\lambda_n$ on $(X^n,T^{(n)})$ by
\[\lambda_n=\int_X \mu_x^{(n)} d \mu (x),\]
where $\mu_x^{(n)}=\mu_x\times\mu_x\times\dotsb\times\mu_x$ ($n$-times) and
$T^{(n)}=T\times T\times\cdots\times T$ ($n$-times).
As $\mu$ is ergodic and has positive entropy, ${\mu}_x$ is non-atomic for $\mu$-a.e. $x\in X$
and
$\lambda_n$ is a $T^{(n)}$-invariant ergodic measure on $X^{n}$
(see \cite[Lemma 5.4]{HY06}).
Since $\mu_x$ is non-atomic for $\mu$-a.e. $x\in X$,
by the Fubini theorem, $\lambda_n(\Delta^{(n)})=0$.
By Theorem~\ref{thm:tau-x}, there exists a disintegration of $\lambda_n$
with respect to the sequence $\{a_k\}$:
\[\lambda_n=\int \tau_{(x_1,x_2,\dotsc,x_n)}d\lambda_n(x_1,x_2,\dotsc,x_n).\]
Consider the continuous function
\[f\colon X^n\to \mathbb{R},\quad (x_1,x_2,\dotsc,x_n)\mapsto \min\{d(x_i,x_j)\colon 1\leq i<j\leq n\}.\]
By Theorem \ref{thm:tau-x}, we have
\begin{align*}
\lim_{N\rightarrow \infty}
\frac{1}{N}\sum_{k=1}^{N} f\Bigl((T^{(n)})^{a_k}(x_1,x_2,\dotsc,x_n)\Bigr)&=
\lim_{N\to\infty}\frac{1}{N}\sum_{k=1}^N\min_{1\leq i<j\leq n} d(T^{a_k}x_i,T^{a_k}x_j)\\
&=\int f d\tau_{(x_1,x_2,\dotsc,x_n)}.
\end{align*}
Since $\lambda_n(\Delta^{(n)})=0$, $\tau_{(x_1,x_2,\dotsc,x_n)}(\Delta^{(n)})=0$
for $\lambda_n$-a.e. $(x_1,x_2,\dotsc,x_n)$.
Note that for any $(x_1,x_2,\dotsc,x_n)\not\in\Delta^{(n)}$,
$f(x_1,x_2,\dotsc,x_n)>0$.
Thus, we have
$\int fd\tau_{(x_1,x_2,\dotsc,x_n)}>0$
for $\lambda_n$-a.e. $(x_1,x_2,\dotsc,x_n)$.

Let $\pi\colon X^n\to X$, $(x_1,x_2,\dotsc,x_n)\mapsto x_1$ be the canonical projection to
the first coordinate.
By \cite[Theorem~4]{GTW00} (see also \cite[Lemma 4.2]{HXY15}), we know that
the Pinsker $\sigma$-algebra of $(X^n,\lambda_n,T^{(n)})$ equals $\pi^{-1}(P_\mu(T))\pmod {\lambda_n}$.
So
\[\lambda_n=\int_X \mu_x^{(n)} d \mu (x)\]
can be also regarded as the disintegration of $\lambda$ over the Pinsker $\sigma$-algebra
of $(X^n,\lambda_n,T^{(n)})$.
By Theorem~\ref{thm:Pinsker-constant-atom}, we have that for $\mu$-a.e. $x\in X$,
$\int f d\tau_{(x_1,x_2,\dotsc,x_n)}$ is a constant for
$\mu_x^{(n)}$-a.e. $(x_1,x_2,\dotsc,x_n)\in X^n$.
More specifically, there exists a Borel subset $X_2$ of $X$ with $\mu(X_2)=1$ such
that for any $x\in X_2$,
\[\lim_{N\to\infty}\frac{1}{N}\sum_{k=1}^N\min_{1\leq i<j\leq n} d(T^{a_k}x_i,T^{a_k}x_j)=\eta_{x,n}\]
for $\mu_x^{(n)}$-a.e. $(x_1,x_2,\dotsc,x_n)\in X^n$ and some constant $\eta_{x,n}>0$.
Put
\[D_{n,\eta}(X,T)=\Bigl\{(x_1,x_2,\dotsc,x_n)\in X^n\colon
\limsup_{N\to\infty}\frac{1}{N}\sum_{k=1}^N\min_{1\leq i<j\leq n} d(T^{a_k}x_i,T^{a_k}x_j)\geq\eta\Bigr\}.\]
Similar to \cite[Lemma 2.4]{HLY14}, $D_{n,\eta}(X,T)$ is a $G_\delta$ subset of $X^n$.
Then for each $x\in X_2$,
$D_{n,\eta_{x,n}}(X,T) \cap \supp(\mu_x)^n$ is a dense $G_\delta$ subset of $\supp(\mu_x)^n$.

As $n\geq2$ is arbitrary, there exists a Borel subset $X_0$ of $X$ with $\mu(X_0)=1$ such that
for every integer $n\ge2$ and $x\in X_0$,
$MP_n(X,T)\cap D_{n,\eta_{x,n}}(X,T) \cap \supp(\mu_x)^n$ is a dense $G_\delta$ subset of $\supp(\mu_x)^n$.
Since $\mu_x$ is non-atomic for $\mu$-a.e.
$x\in X$, we can also require that $\mu_x$ is non-atomic
for every $x\in X_0$. Then for each $x\in X_0$,
$\supp(\mu_x)$ is a perfect closed subset of $X$.
By Mycielski's theorem, there exists a Cantor subset $K_x$ of $\supp(\mu_x)$
such that for every integer $n\geq2$,
$$K_x^n\subset (MP_n(X,T)\cap D_{n,\eta_{x,n}}(X,T))\cup\Delta^{(n)}.$$
Thus, $K_x$ is as required.
\end{proof}

\medskip

Recall that the Pinsker factor of a Kolmogorov system is trivial.
Theorem~\ref{thm-1} yields the following.

\begin{cor}
Let $\{a_k\}$ be a pointwise good sequence satisfying Condition ($*$).
Suppose that $(X,T)$ is a topological dynamical system and
there is an invariant measure $\mu$ such that
$\supp(\mu)=X$ and $(X,\mathcal{B},\mu,T)$ is a Kolmogorov system.
Then there exists a Cantor subset $K$ of $X$ and a sequence $\{\eta_n\}$ of positive numbers
such that for every integer $n\geq2$ and pairwise distinct points $x_1,x_2,\dotsc,x_n$ in $K$,
it holds that
\[\liminf_{N\to\infty}\frac{1}{N}\sum_{k=1}^N\max_{1\leq i<j\leq n} d(T^{a_k}x_i,T^{a_k}x_j)=0\]
and
\[\limsup_{N\to\infty}\frac{1}{N}\sum_{k=1}^N\min_{1\leq i<j\leq n} d(T^{a_k}x_i,T^{a_k}x_j)
\geq \eta_{n}.\]
\end{cor}

\begin{rem}
One of the referees asked the following interesting question:
Does there exist some subsequences
which are bad for mean Li-Yorke chaos in dynamical systems with
positive topological entropy?
In particular, how about the sequence $\{2^n\}$?

It is proved by Bellow in~\cite{B83} that very
fast growing sequences, known as lacunary sequences,
are bad for mean ergodic theorems.
Recall that a positive sequence $\{a_n\}$
 is called \emph{lacunary} if there exists $\lambda > 1$ such
that $\frac{a_{n+1}}{a_n}\geq \lambda$ for all $n\geq 1$.
Clearly the sequence $\{2^n\}$ is lacunary.
Note that the proof of Theorem~\ref{thm-1} strongly depends
on the sequences which are good for pointwise ergodic theorems.
To solve the above question, new technique should be involved.
\end{rem}

\subsection{Non-invertible case}
In this subsection, we generalize Theorem \ref{thm-1} to the non-invertible case.
Let $(X,T)$ be a non-invertible system, meaning that
$X$ is a compact metric space and the map
$T\colon X\to X$ is continuous surjective but not one-to-one.
\begin{thm} \label{thm:non-invertible-case}
Let $(X,T)$ be a non-invertible system and
$\mu\in\mathcal{M}^e(X,T)$ with $h_\mu(T)>0$.
If $\{a_k\}$ is a pointwise good sequence and
the Pinsker $\sigma$-algebra $P_\mu(T)$ is a characteristic $\sigma$-algebra
for the sequence $\{a_k\}$,
then for $\mu$-a.e. $x\in X$, there exists a Cantor subset
$K_x\subset\overline{W^s(x,T)}$ satisfying that
for every integer $n\geq2$ and
pairwise distinct points $x_1,x_2,\dotsc,x_n$ in $K_x$
we have
\[\liminf_{N\to\infty}\frac{1}{N}\sum_{k=1}^N\max_{1\leq i<j\leq n} d(T^{a_k}x_i,T^{a_k}x_j)=0\]
and
\[\limsup_{N\to\infty}\frac{1}{N}\sum_{k=1}^N\min_{1\leq i<j\leq n} d(T^{a_k}x_i,T^{a_k}x_j)
\geq \eta_{x,n},\]
where $\eta_{x,n}$ is a positive constant depending only on $x$ and $n$.
\end{thm}
\begin{proof}
We consider the natural extension $(\tilde{X}, \tilde{T})$ of $(X, T)$;
that is,
\[\widetilde{X}=\{(x_{1}, x_{2}, \cdots)\in X^{\mathbb{N}}\colon Tx_{i+1}=x_{i},\ i\in \mathbb{N}\}\]
and $\widetilde{T}\colon\widetilde{X}\to \widetilde{X}$,
$(x_{1}, x_{2}, \cdots)\mapsto (Tx_{1}, x_{1}, x_{2},\cdots)$ is the shift homeomorphism.
A compatible metric $\tilde{d}$ on $\widetilde{X}$ is defined by
\[\tilde{d}((x_{1}, x_{2}, \cdots), (y_{1}, y_{2}, \cdots))=
\sum_{i=1}^{\infty}\frac{d(x_{i}, y_{i})}{2^{i}}.\]
Let $\pi\colon \widetilde{X}\to X$, $(x_1,x_2,\dotsc)\mapsto x_1$ be the projection to the first coordinate.
Then $\pi\colon(\widetilde{X},\widetilde{T})\to(X,T)$ is a factor map.
As $\mu$ is an ergodic invariant measure of $(X,T)$,
there exists an ergodic invariant measure $\widetilde{\mu}$ of $(\widetilde{X},\widetilde{T})$
such that $\pi(\widetilde{\mu})=\mu$. Clearly, $h_{\widetilde{\mu}}(\widetilde{T})\geq h_{\mu}(T)>0$.
By Theorem~\ref{thm-1}, there is a Borel subset $\widetilde{X}_{0}$ of $\widetilde{X}$
with $\widetilde{\mu}(\widetilde{X}_{0})=1$ such that for each $\tilde{x}\in\widetilde{X}_{0}$,
there exists a Cantor subset $K_{\tilde{x}}\subset\overline{W^s(\tilde{x},\widetilde{T})}$
satisfying that
for every integer $n\geq2$ and pairwise distinct points
$\tilde{x}_1,\tilde{x}_2,\dotsc,\tilde{x}_n$ in $K_{\tilde{x}}$ one has
\[\liminf_{N\to\infty}\frac{1}{N}\sum_{k=1}^N\max_{1\leq i<j\leq n}
 \tilde{d}(\widetilde{T}^{a_k}\tilde{x}_i,\widetilde{T}^{a_k}\tilde{x}_j)=0\]
and
\[\limsup_{N\to\infty}\frac{1}{N}\sum_{k=1}^N\min_{1\leq i<j\leq n}
\tilde{d}(\widetilde{T}^{a_k}\tilde{x}_i,\widetilde{T}^{a_k}\tilde{x}_j)
\geq \eta_{\tilde{x},n}.\]
Let $X_{0}=\pi(\widetilde{X}_{0})$.
Then $X_{0}$ is a $\mu$-measurable set with $\mu(X_{0})=1$.
For any $x\in X_{0}$,
there exists $\tilde{x}\in\widetilde{X}_{0}$ such that $\pi(\tilde{x})=x$.
Let $K_{x}=\pi(K_{\tilde{x}})$. Then $K_{x}\subset\overline{W^s(x,T)}$.
Similar to \cite[Lemma~3.7]{HLY14}, for every integer $n\geq2$ there is a positive constant $\eta_{x,n}$
such that for any pairwise distinct points $x_1,x_2,\dotsc,x_n$ in $K_x$ one has
\[\liminf_{N\to\infty}\frac{1}{N}\sum_{k=1}^N\max_{1\leq i<j\leq n} d(T^{a_k}x_i,T^{a_k}x_j)=0\]
and
\[\limsup_{N\to\infty}\frac{1}{N}\sum_{k=1}^N\min_{1\leq i<j\leq n} d(T^{a_k}x_i,T^{a_k}x_j)
\geq \eta_{x,n}.\]
So $K_x$ is as required.
\end{proof}

\medskip

\section*{Acknowledgments}
The first author was supported in part by NSF of China (grant numbers 11401362 and 11471125).
The second author was supported in part by NSF of China (grant numbers 11371339 and 11571335).
The authors would like to thank professor Wen Huang for numerous discussions
on the topics covered by the paper.
The authors would also like to thank
the anonymous referees for their helpful suggestions
concerning this paper.

\medskip
\end{document}